\documentclass[a4paper,11pt]{amsart}
\usepackage{hyperref,latexsym}
\usepackage{enumerate}

\theoremstyle{plain}
\newtheorem{theorem}{Theorem}[section]
\newtheorem{lemma}[theorem]{Lemma}
\newtheorem{corollary}[theorem]{Corollary}

\theoremstyle{definition}

\theoremstyle{remark}
\newtheorem{remark}{Remark}

\newcommand{\abs}[1]{\left|#1\right|}

\begin{document}

\title[Periodic solutions for a reduction of Benney chain]
      {On Periodic solutions for \\a reduction of Benney chain}

\date{20 March 2008}
\author{Misha Bialy}
\address{Raymond and Beverly Sackler School of Mathematical Sciences, Tel Aviv University,
Israel} \email{bialy@post.tau.ac.il}
\thanks{This paper was started on Island 3 meeting on July 2007.
It is my pleasure to thank the organizers for the support}


\begin{abstract}

  We study periodic solutions for a quasi-linear system, which is the
  so called  dispersionless Lax reduction of the
  Benney moments chain. This question naturally arises in search of
  integrable Hamiltonian systems of the form $ H=p^2/2+u(q,t) $ Our
  main result classifies completely periodic solutions for 3 by 3
  system. We prove that the only periodic solutions have the form of traveling
  waves, so
  in particular, the potential $u$ is a function
  of a linear combination of $t$ and $q$. This result implies
  that the there are no nontrivial cases of existence of the fourth
  power integral of motion for $H$: if it exists, then it is equal necessarily
  to the square of the quadratic one. Our method uses two new
  general
  observations.  The first is the genuine non-linearity of the maximal and minimal
  eigenvalues for the system. The second observation uses the compatibility
  conditions of Gibonns-Tsarev in order to give certain exactness for the system in Riemann
  invariants. This exactness opens a possibility to apply
  the Lax analysis of blow up of smooth solutions, which
  usually does not work for systems of higher order.
\end{abstract}

\maketitle

\section{Introduction}
\label{sec:intro}

Let $ H=p^2/2+u(q,t) $ be a Hamiltonian of a $1,5$-degrees of
freedom system with the potential which is assume throughout this
paper  to be periodic function in both variables. There is a
conjecture attributed to G.Birkhoff saying that the only
integrable plane convex billiards are ellipses. The direct analog
of this conjecture for the Hamiltonian system with 1,5 degrees of
freedom would be the claim that the only integrable Hamiltonian
functions of the form $ H=p^2/2+u(q,t)$ are those having the
potential functions $u$ which are periodic  functions of the form
: $u=u( mq+nt)$. There are several attempts to approach this
problem. Let me refer to the works
\cite{delshams},\cite{bialy1},\cite{bp},\cite{perelomov},\cite{akn}
for various approaches to this circle of questions which remain
outside the discussion of this paper .

In the present note we restrict this question to the search of the
additional integrals which are polynomial with respect to the
momenta variable $p$ with the coefficients which are periodic
functions of $q$ and $t$ . More precisely we want to find all those
potential functions $u(q,t)$ for which there exists an additional
function $F(p,q,t)$ invariant under the Hamiltonian flow (such an
$F$ is called the first integral of motion) which is a polynomial in
the variable $p$ of a given degree, say $n+1$, having all the
coefficients periodic in $q$ and $t$. Write
$$ F(p,q,t)=u_{-1}p^{n+1}+u_0p^n+u_1p^{n-1}+\cdots +u_n,$$
and substitute to the equation of conservation of $F$.
\begin{equation}
 \label{eq:conservation}
 F_t+pF_q-u_qF_p=0
\end{equation}
Equating to zero the coefficients of various powers of $p$, one
easily obtains the following information. The coefficient $u_{-1}$
must be a constant, which will be normalized to be ${1\over n+1}$.
Also $u_0$ must be a constant, which we shall assume to be zero
(this can be achieved by a linear change of coordinates on the
configuration space $\mathbb{T}^2$). Moreover the coefficient $u_1$
satisfies $(u_1)_q=(u)_q$. Therefore, $u_1$ and $u$ will be assumed
to be equal (the addition of any function of $t$ to the potential
$u$ does not change the Hamiltonian equations). Moreover, the column
of the of the rest of the coefficients $U=(u_1 ,.., u_n)^t$ satisfy
the following quasi-linear system of equations.

\begin{equation}
  \label{system}
  U_t+A(U)U_q=0 , \quad
  A(U)=-
  \begin{pmatrix}
  0 & -1 & 0 &  \cdots & 0 & 0 \\[1mm]
  {(n-1)}u_1 & 0 & -1 & 0 & \cdots & 0 \\[1mm]
  \vdots & \vdots & \vdots & \vdots & \vdots & \vdots \\
  2u_{n-2} & 0 & 0 & \cdots & 0 & -1\\
  u_{n-1} & 0 & 0 & \cdots & 0 & 0
  \end{pmatrix}
\end{equation}
Notice that the derivative $F_p$ of $F$  coincides with the
characteristic polynomial of $A(U)$.

In fact the system (\ref {system}) is very well known among
integrable systems community. This is the so called dispersionless
Lax reduction of the moments Benney chain ( see for example
\cite{gk}, \cite{gt}, \cite {pavlov}, \cite {tsarev} and
references there in). There are many beautiful properties of this
reduction. For example, this is a Hamiltonian system of
Hydrodynamics type (in the sense of Dubrovin and Novikov see
\cite{dkn}. Moreover, it has infinitely many additional
conservation laws \footnote{I was told by M.Pavlov that this is a
well known fact among the specialists in integrable systems, see
also \cite{bialy2} where it was rediscovered}. The most important
property for this paper is that the system (\ref{system}) is
diogonalizable, i.e. can be written in the form Riemann invariants
(see (\ref {riemann}) below). Almost nothing is known, however,
about the global existence of smooth solutions for this system. In
the theory of quasi-linear hyperbolic PDEs it is a well known
problem to prove the occurrence of blow up of smooth solutions. It
was performed first by Lax \cite {lax} for 2 by 2 systems
satisfying the so called genuine nonlinearity condition. His
method relies heavily on the possibility to write the 2 by 2
system in the diagonal form, and also on the genuine non-linearity
condition. Lax analysis was performed in \cite{bialy3} for 2 by 2
system of the form (\ref{system}) where it was proved that the
only periodic solutions for that case are constants. For systems
of higher size the original method by Lax does not apply in
general. We will show bellow two new observations concerning our
quasi-linear system. The first is, that for the hyperbolic case,
i.e. the case when all eigenvalues of the matrix $A(U)$ are
distinct and real, it follows that the minimal and maximal
eigenvalues of $A(U)$ are in fact genuinely non-linear in the
sense of Lax (see Corollary \ref{genuine} bellow) . The second key
fact is that the so called Gibbons-Tsarev compatibility system
(the equation (\ref{derivative-4}) of Lemma \ref{lemma1} and
Corollary \ref{exactness}) provides certain "exactness" of the
system (see the equation (\ref{differentiating 3}) and therefore
enables one to perform the Lax original analysis for higher values
of $n$.

 Our main application of this approach in this paper is
the following classification of smooth periodic solutions for the
system.
\begin{theorem}
 \label{main theorem} Let $n=3$. Then the only periodic solution of the
 quasi-linear system (\ref {system}) are the traveling waves
 solutions, where $u_1,u_2,u_3$ do not depend on t.
\end{theorem}

\begin{corollary}
 \label{corollary}
 Let $F=1/4p^4+u_1p^2+u_2p+u_3$ be a  polynomial of degree four
with periodic coefficients which satisfies the equation
(\ref{eq:conservation}). Then there necessarily exists a quadratic
integral, i.e. the energy $H$, and $F$ is a function of $H$.
\end{corollary}
In order to prove this theorem we shall divide between different
regions: strictly Hyperbolic region $\Omega_h$ where all three
eigenvalues of $A(U)$ are real and distinct, Elliptic region
$\Omega_e$, where two eigenvalues are complex conjugate and the
third one is real, and finally the region of degeneracy
$\Omega_0$, where at least two of the eigenvalues collide. It
turns out that in all these regions the behavior of solutions can
be understood completely. The proof of Theorem \ref{main theorem}
is obtained just by patching together the information of these
three cases.
 Organization of
the paper is as follows. In the next Section we shall explain the
two basic observations mentioned above concerning hyperbolic
situation. In Sections \ref{hyperbolic}, \ref{elliptic},
\ref{degenerate} we study the regions $\Omega_h,\Omega_e,\Omega_0$
respectively for the case of $n=3$. In Section \ref{lemma} we prove
formulas for derivatives of the eigenvalues and verify, for the sake
of completeness, the Gibbons-Tsarev compatibility conditions.

\section*{Acknowledgements}
This paper was started during the Island 3 meeting on Integrable
systems. It is my pleasure to thank the organizers and the
participants for excellent conference. It was especially important
for me to discuss theory of  the systems of Hydrodynamic type with
Maxim Pavlov, he explained to me many facts about Benney chain. It
is my pleasure to thank my colleagues Steve Shochet and Misha Sodin
for very stimulating discussions and help.

\section { Main Observations }
Let me denote by $\lambda _1,..,\lambda _n$ the roots of the
polynomial $F_p$. And let $r_i=F(\lambda _i)$ be the corresponding
critical values. The starting point for us is a beautiful classical
theorem by MacLane, stating that the mapping $(u_1...u_n) \mapsto
(r_1,..,r_n)$ is in fact a global diffeomorphism between the domain
of strict Hyperbolicity (that is the domain of all
$U=(u_1,\dots,u_n)^t$ where all the roots of the polynomial $F_p$
are real and distinct) with the domain of all possible critical
values in $\mathbb{R}^n$ that is of all those $(r_1,\dots,r_n)$ such
that the differences $(r_k-r_{k+1})$ have the sign $(-1)^{k+n}$ for
all $k=1,\dots,n-1$ (we refer to \cite{eremenko} for E.B. Vinberg's
proof of this theorem and further results and discussions).
According to this theorem $(r_1,..,r_n)$ can be taken as regular
global coordinates in this domain.

The importance of these coordinates for our system follow from the
following computation. Substitute $p=\lambda _i$ into the equation
(\ref{eq:conservation}). One gets the following diagonal system on
the variables $r_i$ (they are called Riemann invariants)
\begin{equation}
 \label{riemann}
 (r_i)_t+\lambda_i(r_1, \dots ,r_n) (r_i)_q=0 ,\quad  i=1,..,n
\end{equation}

The derivatives of the roots $\lambda_i$ with respect to the
critical values $r_i$ satisfy the following relations.
\begin{lemma}
\label {lemma1} The following formulas hold true
\begin{enumerate}[(a)]

\item $\partial _{r_i} \lambda_i=-\frac{1}{F_{pp}(\lambda_i )}
\sum_{k=1,k \neq i}^{n}\frac {1}{\lambda _i- \lambda _k} $,
\label
{derivative-1} \vspace{2mm}
\item $\partial _{r_i} \lambda_j=-\frac{1}{F_{pp}(\lambda_i)}
\frac {1}{\lambda _j-\lambda_i}, \quad i \neq j$, \label
{derivative-2} \vspace{2mm}
\item $ F_{pp}(\lambda_i) u_{r_i}=1$,\label {derivative-3}
\vspace{2mm}
\item $u_{{r_i}{r_k}}=\frac{2u_{r_i}u_{r_k}}{(\lambda _i- \lambda
_k)^2}$.
\label {derivative-4}
\end{enumerate}
\end{lemma}

\begin{remark}

The condition (\ref{derivative-4}) is in fact the so called Gibbons
Tsarev compatibility condition (see \cite{gt})(note there is a
missprint in their formula-the factor $2$ is missing). I didn't find
however the formula (\ref{derivative-1}) in any paper on the
subject. We suggest the proof of all of them in Section \ref{lemma}
in a very short way.
\end {remark}

\begin{corollary}
\label {genuine} In the strictly hyperbolic region $\lambda_1 <
\lambda_2 < \dots < \lambda_n$ the maximal and minimal eigenvalues
are genuinely nonlinear in the sense of Lax:
$$\partial_{r_1} \lambda_1 \neq 0,\partial _{r_n} \lambda_n \neq 0 $$
\end{corollary}

Another important consequence of the Gibbons-Tsarev conditions is
the following
\begin{corollary}
\label{exactness}
 In the Hyperbolic region introduce the functions
 $$G_i=-\frac {1}{2}\log\abs {u_{r_i}}=\frac
{1}{2}\log \abs {F_{pp}(\lambda_i)}$$ Then it follows from the
lemma that
$$\partial_{r_j}G_i=-\frac {u_{r_i r_j}}{2u_{r_i}}=-\frac{u_{r_j}}{(\lambda _i- \lambda
_j)^2}=\frac{(\lambda_i)_{r_j}}{\lambda_i-\lambda_j}$$.

\end{corollary}

In order to perform the blow up analysis for our system we shall
differentiate the quantities $w_i, w_i=(r_i)_q$ along the integral
curves of the family $\lambda_i$. They are, by definition, the
integral curves of the equation $\dot{q}+\lambda_i(q,t)=0$ on
$\mathbb{T}^2$.
 Let $v_i=(1,\lambda_i(q,t))$ be the $i$-th vector field and let
 $L_{v_i}=\partial_t + \lambda_i \partial_q$ denotes the Lie derivative along the field
 $v_i$.
 Differentiating with respect to $q$ the i-th equation of (\ref
 {riemann}) one gets the following
 \begin {equation}
 \label{differentiating}
 L_{v_i}(w_i)+w_i^2(\lambda _i)_{r_i}+w_i \sum_{j \neq i} (\lambda
 _i)_{r_j}(r_j)_q=0
 \end {equation}
Notice that by definition:
$$L_{v_i}r_j=(r_j)_t+\lambda_i(r_j)_q$$
 Subtract from this expression the
$j$-th equation of (\ref{riemann})
$$0=(r_j)_t+\lambda_j(r_j)_q$$
 one verifies that $$(r_j)_q=\frac {L_{v_i}{r_j}}{\lambda_i-\lambda_j}.$$
 Substitution of this expression into (\ref{differentiating}) leads to:
 \begin {equation}
 \label{differentiating 2}
 L_{v_i}(w_i)+w_i^2(\lambda _i)_{r_i}+w_i \sum_{j \neq i} (\lambda
 _i)_{r_j}\frac {L_{v_i}{r_j}}{\lambda_i-\lambda_j}=0
 \end {equation}
Therefore it follows from the last (Corollary \ref{exactness}) that
(\ref{differentiating 2}) can be rewritten in the following way
$$L_{v_i}(w_i)+w_i^2(\lambda
_i)_{r_i}+w_i\sum_{j \neq i}(G_i)_{r_j}L_{v_i}r_j =0.$$ Therefore,
taking into account the i-th equation of (\ref{riemann}) we get:
\begin{equation}
\label{differentiating 3} L_{v_i}(w_i)+w_i^2(\lambda
_i)_{r_i}+w_iL_{v_i}G_i=0
\end{equation}
Multiplying by $\exp{G_i}$ this equation one rids of the linear term
as follows:
\begin{equation}
L_{v_i}((\exp{G_i})(w_i))+ (\exp{(-G_i)}(\lambda
_i)_{r_i})(\exp({2G_i})w_i^2)=0
\end{equation}
Using the explicit expression for $G_i$  of Corollary
\ref{exactness} and denoting
$$z_i=\abs {F_{pp}(\lambda_i)}^{1/2}w_i=\abs
{F_{pp}(\lambda_i)}^{1/2}(r_i)_q, \quad K_i=\abs
{F_{pp}(\lambda_i)}^{-1/2}(\lambda_i)_{r_i}$$ we get the following
equation, which is crucial for the analysis of the blow up of the
solution.
\begin{equation}
\label{blow up}
 L_{v_i}z_i+K_iz_i^2=0, \quad i=1,\dots,n
\end{equation}
As an immediate consequence of this equation we state the following
\begin{theorem}
Let $U=(u_1,\dotsc,u_n)^t$ be a periodic solution of the
quasi-linear system (\ref{system}) corresponding to the strictly
Hyperbolic regime, i.e. all eigenvalues are real and distinct:
$\lambda_1<\lambda_2< \dots<\lambda_n$. Then the Riemann invariants
 $r_1$ and $r_n$ corresponding to the minimal and maximal eigenvalues are
constants.
\end{theorem}

\begin{proof}
This fact follows immediately from the equation (\ref{blow up}).
Indeed, in the region of strict hyperbolicity the
$F_{pp}(\lambda_i)$ does not vanish and so, by the genuine
non-linearity of $\lambda_1$ and $\lambda_n$ the functions $K_1,K_n$
are bounded away from zero. Then it follows from the explicit
formula for the solutions of (\ref{blow up}) that the only solution
which does not explode in a finite time is $z_1,z_n=0$. Thus
$(r_1)_q=(r_n)_q=0$ and the equations (\ref{riemann}) imply that
$r_1$ and $r_n$ must be constants. This yields the result.
\end{proof}
A refinement of this argument is the content of the next section on
the Hyperbolic region, for 3 by 3 system. In what follows we shall
assume that $n=3$, i.e. the system is 3 by 3. Let me denote by
$\Omega_h$ be the region of strict Hyperbolicity and $\Omega_e$ be
the region, where there are two complex conjugate eigenvalues for
$A(U)$. The complement, $\mathbb{T}^2-(\Omega_h \cup \Omega_e)$ is
the set of those points $(q,t)$ where the matrix $A(U)$ has at least
two equal eigenvalues . We shall denote this set $\Omega_0$. And
finally, $\Omega_{00}$ will denote the set of maximal degeneration,
i.e. where all three eigenvalues are equal and thus equal to zero
(since the sum of all the three eigenvalues vanishes).

\section{Hyperbolic region $\Omega_h$.}
\label{hyperbolic} Before stating the main result of this section,
let me rewrite the formulas of
 the Lemma \ref{lemma1} for the case of 3 by 3 system.

 In this case since $\lambda_1+\lambda_2+\lambda_3=0$ we get the
 following simplifications
 $$ (\lambda_i)_{r_i}=-\frac{3\lambda_i}{F_{pp}(\lambda_i)^2}, \quad
 K_i=-\frac{3\lambda_i}{\abs{F_{pp}(\lambda_i)}^{5/2}}, \quad
 F_{pp}(\lambda_i)= \prod_{j \neq i}(\lambda_i-\lambda_j)$$

For the case $n=3$ we have the following refinement of the theorem
of the previous section.

 \begin{theorem}
\label{th:blow up}Let $U=(u_1,u_2,u_3)^t$ be a periodic solution of
the system (\ref{system}), and let $\Omega_h \subseteq \mathbb{T}^2$
be the domain of strict Hyperbolicity. Then
  the
 Riemann invariants $r_1, r_3$ are constants in every connected component of
 the Hyperbolic domain $\Omega_h$.
 \end{theorem}

 \begin{proof} We give the proof for $r_1$, the other case is analogous.
 The first step of the proof is the fact that
 the derivatives $(r_i)_q$ are bounded functions on the domain
 $\Omega_h$. Indeed, by  definition,
 $$r_i=F(\lambda_i,q,t)\Rightarrow
 (r_i)_q=F_q(\lambda_i, q,t)$$
 and thus by the periodicity of the
 coefficients of the polynomial $F$, all roots $\lambda_i$ of the derivative $F_p$  are
 bounded and so are the derivatives $(r_i)_q$.
Consider the integral curves of the $\lambda_1$-family in the domain
$\Omega_h$. Suppose that such a curve approaches the boundary of
$\Omega_h$, then $F_{pp}\rightarrow 0$ while $(r_i)_q$ stays
bounded. Therefore, it follows that
 $$z_i=(r_i)_q|F_{pp}(\lambda_i)|^{1/2}\rightarrow 0.$$
  I claim that then $z_1$ equals zero identically in
 $\Omega_h$. If, for example, $z_1$ is positive at a point
 $(q_0,t_0)$
 then by the equation (\ref{blow up}) and the fact that $K_1$ is
 positive, we have that $z_1$ is a decreasing function of time and thus
 in the backward time along the integral curve cannot approach
 the boundary, because on the boundary $z_1$ vanishes. On the other hand,
 if the integral curve stays inside $\Omega_h$ forever in the backward
 time, then the function $K_1$ stays
 bounded away from zero, and therefore the blow up of the solution must
 occur
 in a finite (backward) time. Thus $z_1$ can not be positive.
 The opposite case, when $z_1(q_0,t_0)$ can not be negative is completely analogous.

 This argument in a more precise form looks as follows.
Denote by $M_1,M_2$
 positive constants  such that
 $$\abs{(r_1)_q}< M_1 , \quad \sqrt{\abs { F_{pp}(\lambda_i)}}<M_2.$$
 Then for any backward time along the integral curve
 the
 monotonicity of $z_1$ implies that

$$M_1\sqrt{\abs { F_{pp}(\lambda_1)}}>(r_1)_q \sqrt{\abs { F_{pp}(\lambda_1)}}=z_1\geq
z_1(q_0,t_0)$$ and therefore
 $$\sqrt{\abs { F_{pp}(\lambda_1)}}>\frac {z_1(q_0,t_0)}{M_1}.$$ In
 addition,
$ \lambda_1$ can not be too close to zero. Indeed, if $0\le
-{\lambda_1}<a$, then by the zero sum condition also
$\abs{\lambda_2}<a$ and $0\le{\lambda_3}<2a$. Then one would get
$$\sqrt{\abs { F_{pp}(\lambda_i)}}=((\lambda_2-\lambda_1)(\lambda_3-\lambda_1))^{1/2}<\sqrt{6}
 a.$$
So together with the previous estimate this implies that
$\abs{\lambda_1}\geq \frac {z_1(q_0,t_0)}{\sqrt{6} M_1}$. Then
$$K_1=
\frac{3\abs{\lambda_1}}{((\lambda_2-\lambda_1)(\lambda_3-\lambda_1))^{5/2}}
\geq \frac {3z_1(q_0,t_0)}{\sqrt{6}M_1M_2^5}>0$$ So $K_1$ is bounded
away from zero, and again by the explicit formula for the solution
of (\ref{blow up}) it explodes in a finite backward time. This
proves the claim that $z_1$ vanishes identically in $\Omega_h$. Thus
$z_1,(r_1)_q\equiv0$, and so by the equations (\ref{riemann}) $r_1$
must be constants.
\end{proof}
The next theorem  describes completely the solutions of the system
in the Hyperbolic region $\Omega_h$ .
\begin{theorem}
\label {Omegah}Either the solution $U=(u_1,u_2,u_3)^t$ is a constant
solution for (\ref{system}) on $\mathbb{T}^2$, or the region
$\Omega_h$ is a union of strips on the torus parallel to the $t$
-axes and the following relations hold
$$u_1=u_1(x), u_2 \equiv0, u_3=u_1^2+const,$$
so that the polynomial $F$  equals (up to a constant) in $\Omega_h$
 to the square of the Hamiltonian
$$F=\left(\frac{p^2}{2}+u_1(q)\right)^2 +const.$$
\end{theorem}

\begin{proof}
Let me note, that since $r_i$ are successive critical values of the
polynomial $F_p$, then everywhere in $\Omega_h$ holds $r_2>r_1,
r_2>r_3$. On the boundary $\partial\Omega_h$ two of the eigenvalues
collide, say $\lambda_2$ collides with $\lambda_1$ (or with
$\lambda_3$ ,or maybe both), and hence $r_2=r_1$ ( or $ r_2=r_3)$.
By the previous theorem $r_1,r_3$ are constants in $\Omega_h$, and
in addition $r_2$ has constant values along the integral curves of
the $\lambda_2$-family. Then it follows, that non of these curves
can approach the boundary, because otherwise this would give
$r_1=r_2$ in the inner points. Moreover, the function $r_2$
satisfies the equation
$$(r_2)_t+\lambda_2(r_2)(r_2)_q=0,$$
where $\lambda_2(r_2)$ depends only on $r_2$, since $r_1,r_3$ are
constants in $\Omega_h$. Therefore, the characteristics of this
equation are the straight lines in the $(q,t)$-plane, and thus there
are two possible cases. The first case is that there exist two
intersecting straight lines of the family. Then $r_2$ has to be
constant everywhere ( since $r_2$ has constant values along
characteristics, and any straight line intersects at least one of
the two intersecting characteristics). So in this case all
$r_1,r_2,r_3$ are constants everywhere in $\Omega_h$, and hence also
the coefficients of the polynomial $u_1,u_2,u_3$ are constants. In
such a case $\Omega_h$ must the whole $\mathbb{T}^2$.

In the second case all the characteristic straight lines are
parallel with the same $\lambda_2=\mu=const$ , and thus the domain
$\Omega_h$ in this case is a union of parallel strips with the slope
$\mu$. Next we claim that if the solution $U=(u_1,u_2,u_3)^t$ is not
constant in $\Omega_h$, then $\mu=0$. To see this, let me recall
that
$$F_p(\mu)=\mu^3+2u_1\mu+u_2=0,$$ and therefore
\begin{equation}
\label{vyrazhenie} u_2=-2 \mu u_1-\mu^3
\end{equation}
Substituting (\ref{vyrazhenie}) into the first equation of the
system (\ref {system}) we have
$$(u_1)_t=-(u_2)_q=2\mu (u_1)_q$$
In addition we have that along any characteristic straight line of
the family $\lambda_2$ the values of $r_1,r_2,r_3$ are constants,
and then also $\lambda_1,\lambda_2,\lambda_3,u_1,u_2,u_3$, because
$r_1,r_2,r_3$ are genuine coordinates. In order to prove the claim
let us assume that on the contrary $\mu\neq0$.
 Then $u_1$ has to be globally constant in
$\Omega_h$ because it satisfies the following two equations
\begin{equation}
\label{mu} (u_1)_t-2\mu (u_1)_q=0, \quad (u_1)_t+\mu (u_1)_q=0
\end{equation}
But then, by (\ref{vyrazhenie}), also $u_2$ is constant in
$\Omega_h$. Therefore $\lambda_1,\lambda_2,\lambda_3$ are all
constants, since the polynomial $F_p$ has constant coefficients.
Also since $u_3=r_1-F(\lambda_1)$ then also $u_3$ must be a
constant. So the solution is in fact a constant solution
contradicting the assumption of the claim.

Thus, we get that $\mu=0$. In this case all characteristics of the
second family are parallel to the $t$-axes and the region $\Omega_h$
is the union of strips parallel to the $t$-axes. Moreover by
(\ref{mu}) $u_1, u_3$ have to be the functions on $q$ only and by
(\ref{vyrazhenie}) $u_2\equiv0$.

 Therefore $F_p=p(p^2+2u_1)$, and so
 $$F=\frac{1}{4}p^4+u_1p^2+u_3=\left(\frac{p^2}{2}+u_1\right)^2+u_3-u_1^2,$$
 and since $r_1,r_3$ are constants, then $u_3-u_1^2=constant$,
 and we are done. Notice that on the boudary of $\Omega_h$,
 $u_1$ vanishes and so $\partial \Omega_h\subseteq\Omega_{00}$.
 \end{proof}
 \section{Elliptic region $\Omega_e$}
 \label{elliptic}
 Consider now the region $\Omega_e$ where the polynomial $F_p$
 has two complex conjugate roots, say $\lambda_{1,2}=\alpha\pm
 i\beta$  with $\beta>0$ in $\Omega_e$, and
 $\lambda_3\in\mathbb{R}$. In this case $r_{1,2}$ are also complex
 conjugate, say $r_{1,2}=v\pm iw$ and $r_3$ is real.
 Notice, that for the points of the boundary of $\Omega_e$ we have
 $\lambda_1=\lambda_2=-\lambda_3/2$ are real, $r_1=r_2$ are real also, and so  $\beta=0,w=0$.
 For the region $\Omega_e$ we have the same description of solutions as in the strictly
 Hyperbolic domain, but for
 completely different reasons.
 \begin{theorem}
\label{omegae} Either $U=(u_1,u_2,u_3)^t$ is a constant solution for
the system (\ref{system}) on the whole $\mathbb{T}^2$ , or the
region $\Omega_e$ is a union of strips parallel to the $t$-axes on
the torus and and the following relations hold
$$u_1=u_1(q), u_2 \equiv0, u_3=u_1^2+const$$
So that the polynomial $F$  equals (up to a constant) in $\Omega_e$
to the square of the Hamiltonian
$$F=\left(\frac{p^2}{2}+u_1(q)\right)^2 +const.$$
\end{theorem}

\begin{proof}
The Riemann invariants $r_{1,2}=v\pm iw$ satisfy the equations
(\ref{riemann}), which are equivalent to the following elliptic
system on their real and imaginary parts.
\begin{align*}
v_t+\alpha v_q-\beta w_q=0\\
w_t+\beta v_q +\alpha w_q=0.
\end{align*}
This is an elliptic system, and therefore, since $w$ vanishes on the
boundary, then by the strong maximum principle the function $w$ must
vanish identically in the whole $\Omega_e$. Substituting back to the
elliptic system we get that $v$ is a constant everywhere $\Omega_e$.
Therefore, $\lambda_{1,2}$ are  roots of the polynomials $F-v$ and
of $F_p$.  Then,
$$F=\frac{1}{4}(p-\lambda_1)^2(p-\lambda_2)^2+v=
\frac{1}{4}((p-\alpha)^2+\beta^2))^2+v.$$ Moreover, $\alpha$
vanishes because $F$ does not contain the cubic terms. Therefore
$$F=(p^2/2+\beta^2/2)^2+v,\quad  \lambda_{1,2}=\pm i\beta, \lambda_3=0$$
This situation can be completely analyzed, because in this case the
quadratic function $\tilde{F}=p^2/2+\beta^2/2$ is the conserved
quantity. But then, by the equation (\ref{eq:conservation}) for $
\tilde{F}$ we have:
\begin{align*}
(\beta^2)_t=0\\
(\beta^2/2-u_1)_q=0.
\end{align*}
Notice, that on the boundary of $\Omega_e$ we have
$\beta=u_1=u_2\equiv0.$ By the first equation $\beta$ vanishes on
any line ${q=const}$ whenever it crosses the boundary, and it is a
constant on any line which lies entirely inside $\Omega_e$. This
yields immediately that $\Omega_e$ is in fact union of strips
parallel to the ${t}$ axes and $\beta=\beta(q)$ and also
$\beta^2/2-u_1\equiv0$ because on the boundary both $\beta$ and
$u_1$ vanish. So we proved
$$F=(p^2/2+u_1(q))^2+const$$ and this completes the proof
\end{proof}
\section{ Degenerate regions $\Omega_0, \Omega_{00}$}
\label{degenerate}
 It follows from the description of $\Omega_e$
and $\Omega_h$ that the degenerate region $\Omega_0$ is a union of
strips parallel to the $t$ axes and moreover every point of the
boundary of each strip belongs, in fact to $\Omega_{00}$. Then, it
follows that $\Omega_0-\Omega_{00}$ is an open set. We claim next
that the degeneration is maximal everywhere, i.e.
$\Omega_0=\Omega_{00}$. Indeed, consider an integral curve of the
$\lambda_3$-family lying inside this open set
$\Omega_0-\Omega_{00}$. Then $r_3$ has constant values along the
curve, and therefore it can not approach the boundary of the
region $\Omega_0-\Omega_{00}$, since otherwise inside one would
get $r_1=r_2=r_3$ inside the region, which contradicts the
assumptions. Thus the whole characteristic stays inside the region
$\Omega_0-\Omega_{00}$. But then,  exactly as above, in the proof
of the Theorem \ref{th:blow up}, we have that the derivative
$(r_3)_q$ must explode in a finite time unless it vanishes. And
therefore, in the whole region $\Omega_0-\Omega_{00}$, $r_3$ is a
constant and therefore $\Omega_0=\Omega_{00}$, and we are done.

\section {Proof of the Lemma  $\ref{lemma1}$}
\label{lemma}
 It is rather simple to prove (\ref{derivative-1}) and
(\ref{derivative-2}) of the lemma. By the definitions of $\lambda_i$
and $r_i$ we have
$$F_p(\lambda_i)=0 , \quad F(\lambda_i)=r_i$$
Differentiate these two formulas with respect to $r_j$. Notice that
the roots $\lambda_i$ depend on $r_j$ and also the coefficients
$u_i$ of the polynomial $F $. We get
\begin{equation}
\label{Fpp}
 F_{pp}(\lambda_i)(\lambda_i)_{r_j}+F_{p r_j}(\lambda_i)=0 , \quad
 F_{r_j}(\lambda_i)=\delta_{ij}
\end{equation}
Then write the following identity
\begin{equation}
\label{identity} F_{r_i}(p)(p-\lambda_i)=(u_1)_{r_i}F_p(p)
\end{equation}
which becomes clear, if one notice that on both sides there are
polynomials of the same degree $n$ with the same leading coefficient
$(u_1)_{r_i}$, and both having $\lambda_1,\dots,\lambda_n$ as the
roots (the right hand side-just by definition, and the left hand
side-by (\ref{Fpp})).

Differentiate this identity with respect to $p$ to obtain
$$(p-\lambda_i)F_{pr_i}+F_{r_i}=(u_1)_{r_i}F_{pp}$$
Substitute in this formula $p=\lambda_i$ one arrives to
(\ref{derivative-3}) of the lemma immediately. Substitute
$p=\lambda_j$ and take into account(\ref{Fpp}), then one proves
(\ref {derivative-2}).

In order to prove (\ref{derivative-1}) let me use the relation
$F_{r_j}(\lambda_i)=\delta_{ij}$ of (\ref{Fpp}) in order to
conclude, that $ F_{r_i}$ is, in fact, identical to the $i$-th
Lagrange interpolation polynomial:
$$
F_{r_i}=\prod_{s\neq
i}\frac{(p-\lambda_s)}{(\lambda_i-\lambda_s)}=l_i(p)
$$
Differentiate this identity with respect to $p$ at the point
$\lambda_i$:
$$
F_{pr_i}(\lambda_i)=(l_i)_p(\lambda_i)=\sum_{s \neq
i}\frac{1}{(\lambda_i-\lambda_s)}
$$
Using (\ref{Fpp}) we obtain
$$(\lambda_i)_{r_i}=-\frac{F_{pr_i}(\lambda_i)}{F_{pp}}
=-\frac{1}{F_{pp}(\lambda_i)}\sum_{s \neq
i}\frac{1}{(\lambda_i-\lambda_s)}
$$
This gives the proof of (\ref{derivative-1})

In order to derive (\ref{derivative-4}) write (\ref{derivative-2})
in the form
$$(\lambda_j)_{r_i}=u_{r_i}/(\lambda_i-\lambda_j)$$
Following Gibbons, Tsarev (\cite {gt}) differentiate this formula
with respect to $r_k$ to get
\begin{equation}
\label{ik}
(\lambda_j)_{r_ir_k}=\frac{u_{r_ir_k}}{(\lambda_i-\lambda_j)}-
\frac{u_{r_i}u_{r_k}}
{(\lambda_i-\lambda_j)(\lambda_k-\lambda_i)(\lambda_k-\lambda_j)}
\end{equation}
Now change the order of the indices $i$ and $k$ in (\ref{ik}) to
have
\begin{equation}
\label{ki}
(\lambda_j)_{r_kr_i}=\frac{u_{r_kr_i}}{(\lambda_k-\lambda_j)}-
\frac{u_{r_k}u_{r_i}}
{(\lambda_k-\lambda_j)(\lambda_i-\lambda_k)(\lambda_i-\lambda_j)}
\end{equation}
Subtract now (\ref{ik}) and (\ref{ki}). One gets
(\ref{derivative-4}) of the lemma.

\section{Concluding remarks and questions}
1. It would be very interesting to generalize the analysis presented
here for 3 by 3 system to the case of higher orders. In general,
there are much more cases of degenerations of eigenvalues for higher
$n$. This makes the analysis much harder. However, the case of
strictly Hyperbolic solutions seems to be tractable.

2. Besides the question on  periodic solutions for the quasi-linear
system,there  is interesting question to understand the global
existence of smooth (not necessarily periodic in time) solutions,
having Hyperbolic initial data. In other words assume, that we are
given for $t=0$ the initial condition, say periodic functions in $q$
$U_0(q)=(u_{10}(q),\dots,u_{n0}(q))$, such that the eigenvalues of
the matrix $A(U_0)$ are real and distinct. Now switch on the
dynamics. The question is if there are solutions existing for all
times, or blow up of the solutions can be established.

3. One of the key ingredients of the proof in this paper was the
observation that the (semi-)Hamiltonian property of the quasi-linear
system allows one to push forward the Lax analysis of the blow up of
solutions. It would be interesting to know, what can be said in this
perspective for other reductions of the Benney chain.

4. It was proved by Tsarev (\cite{tsarev}), that in principle, the
solutions for the system can be obtained by the so called
generalized hodograph method. This method however, is very implicit
and therefore, it is not clear to me how it can be used in order to
answer the question of long time existence.


\end{document}